\documentclass[12pt, a4paper]{amsart}
\usepackage{amscd}
\usepackage{amsmath}
\usepackage{latexsym}
\usepackage{graphicx}
\usepackage{amsfonts}
\usepackage{amssymb}

\renewcommand\atop[2]{\genfrac{}{}{0pt}{}{#1}{#2}}
\newcommand{\R}{\mathbb{R}}
\newcommand{\C}{\mathbb{C}}
\renewcommand{\H}{\mathbb{H}}
\newcommand{\Z}{\mathbb{Z}}

\newcommand{\cA}{\mathcal{A}}
\newcommand{\cE}{\mathcal{E}}
\newcommand{\eps}{\varepsilon}
\newcommand{\op}{\textrm{op}}
\def\oo{\otimes}
\numberwithin{equation}{section}
\theoremstyle{plain}
\newtheorem{theorem}[equation]{Theorem}
\newtheorem{corollary}[equation]{Corollary}

\newtheorem{lemma}[equation]{Lemma}
\theoremstyle{definition}

\newtheorem{variant}[equation]{Variant}
\newtheorem{examples}[equation]{Examples}
\newtheorem{example}[equation]{Example}

\newtheorem{defn}[equation]{Definition}

\theoremstyle{remark}

\newtheorem*{remm}{Remark}

\def\map#1{{\buildrel #1 \over \longrightarrow}}

\newcommand{\Hom}{\operatorname{Hom}}

\newcommand{\GA}{$G$--$\cA$\ }

\begin{document}
\title[Twisted $K$-theory and Grothendieck-Witt groups]%
{Twisted $K$-theory, Real $A$-bundles and Grothendieck-Witt groups}
\date{\today }
\author{Max Karoubi}
\address{Universit\'e Denis Diderot Paris 7, Institut Math\'ematique de
Jussieu -- Paris rive gauche}
\email{max.karoubi@gmail.com}
\urladdr{http://webusers.imj-prg.fr/~max.karoubi}

\author{Charles Weibel}
\address{Math.\ Dept., Rutgers University, New Brunswick, NJ 08901, USA}
\email{weibel@math.rutgers.edu}
\urladdr{http://math.rutgers.edu/~weibel}
\thanks{Weibel was supported by NSA and NSF grants,} 

\begin{abstract}
We introduce a general framework to unify several variants of
twisted topological $K$-theory. 
We focus on the role of finite dimensional real simple algebras 
with a product-preserving involution, 
showing that Grothendieck-Witt 
groups provide interesting examples of twisted $K$-theory. 
These groups are linked with the classification of algebraic 
vector  bundles on real algebraic varieties.
\end{abstract}

\maketitle

In our recent paper \cite{KWpre}, we compare the Witt group of an
algebraic variety over $\R$ with a purely topological invariant
we called $WR(X)$, associated to a space $X$ with involution.
In our comparison, $X$ is the underlying space of 
complex points of the algebraic variety, and the involution is
induced by complex conjugation.  
We are able to compute $WR(X)$ by comparing
it to classical equivariant topological $K$-theory \cite{Segal},
Atiyah's Real $K$-theory $KR(X)$ \cite{A},
and other familiar invariants.

The Witt group of skew-symmetric forms is approximated in a 
similar way in \cite{KWpre}, using another topological invariant. 
Surprisingly, the computation of this invariant is more subtle,
as it is linked with equivariant twisted $K$-theory 
(in the sense of \cite{DK} and \cite{AS}). 

This paper gives a more systematic study of this equivariant 
twisted $K$-theory, a variant we think is of independent interest.
Our emphasis will be on examples linked with finite dimensional 
simple $\R$-algebras and Grothendieck-Witt groups.

The first section of this paper develops the basic notions of equivariant
twisted $K$-theory in a geometric way, 
adapting many classical arguments of topological $K$-theory to 
vector bundles which are modules over an algebra bundle. 
We don't claim too much originality here.
Much of the recent theory has already been developed in an ad hoc way
(see for instance \cite{AS2}, \cite{Julg}, \cite{Crocker-Raeburn} or
the 2014 thesis of El-ka\"ioum Moutuou \cite{Moutuou}).

In the second section, we study specific examples linked with Real
vector bundles. We follow Atiyah's viewpoint \cite{A}, replacing $\C$
with $\R$-algebras with involutions. The most important examples 
for us are ``balanced'' algebras, such as Clifford algebras,
for which we can provide a simpler description of the theory.

In the last section we link twisted $K$-theory to the
Grothendieck-Witt groups of symmetric and skew-symmetric bilinear forms on 
Real $\cA$-bundles. In particular, we show that the
Grothendieck-Witt group of skew-symmetric bilinear forms is isomorphic
to a ``twisted $KR$-group'' associated to the quaternion algebra $\H$.
This group is different from the group of symplectic bundles 
defined by J.\,Dupont \cite{Dupont} in another context.

\smallskip

We are particularly grateful to Jonathan Rosenberg, who made 
important comments on the first version of this paper. 

\section{$G$--$\cA$ bundles}

In this section we present a variant of ($G$-equivariant) twisted $K$-theory
associated to an algebra bundle $\cA$ on a $G$-space $X.$
We begin with a quick review of the non-equivariant theory.

Let $A$ be a fixed Banach $\R$-algebra with unit and $X$ a paracompact space.
We suppose given a locally trivial bundle $\cA$ of Banach $\R$-algebras
on $X$, with fibers isomorphic (non-canonically) to  $A$; the
structure group of continuous algebra automorphisms has the 
compact-open topology (also called the \emph{strong} topology).

By an {\it $\cA$-bundle} we mean a locally trivial bundle $E$ of (left)
$\cA$-modules on $X$ such that $\cA\!\times_X\! E\to E$ is continuous,
and each fiber $E_x$ is a finitely generated projective $\cA_x$-module.
Morphisms $E\to F$ are bundle maps whose fibers are module maps.
The $\cA$-bundles form an additive category and we write $K^{\cA}(X)$ 
for the Grothendieck group of $\cA$-bundles on $X$.
When $\cA$ is a trivial algebra bundle,  an $\cA$-bundle is just a
(classical) {\it $A$-bundle}, and we write $K^A(X)$ for $K^{\cA}(X)$.

The simplest non-trivial examples arise when $\cA$ is an algebra
bundle with fiber $A=M_n(\C)$ and $X$ is a finite CW complex. 
In this case, the Morita equivalence classes of possible $\cA$
are classified by their class in the topological Brauer group of $X$,
which by a theorem of Serre (\cite[1.7]{Dix}) 
is the torsion subgroup of $H^3(X;\Z)$. Historically, these groups 
$K^{\cA}(X)$ were the first examples of twisted $K$-theory (see
\cite{DK}, \cite{Ktwist}). The theory has since developed 
to include the class in $H^3(X;\Z)$ as part of the data
(see \cite{Rosenberg}),
with the awareness that the generalization to modules over a bundle of
Banach algebras gives rise to interesting applications in physics,
as mentioned for example in \cite{Witten}.

Now let $G$ be a compact Lie group acting on $X$. As usual, we say that
$G$ acts on a bundle $E$ if $G$ acts on $E$ compatibly with the
structure map $E\to X$, in the sense that multiplication by $g\in G$ sends
$E_x$ to $E_{gx}$.  If $\cA$ is a bundle of Banach algebras 
with fiber $A$, we say that $G$ acts on $\cA$ if $G$ acts on the underlying
bundle of $\cA$ by algebra isomorphisms, i.e., if multiplication
by any $g\in G$ is an algebra isomorphism $\cA_x\to\cA_{gx}$ for each $x\in X$.
The notion of a \GA bundle is somewhat related to
Fell's Banach $*$-algebraic bundles over $G$ \cite{Fell}.

\begin{defn}\label{def:GA}
Let $E$ be an $\cA$-bundle whose fiber is a 
finitely generated projective
$A$-module. We say that $E$ is a $G$--$\cA$ bundle on $X$
if $G$ acts on $E$ (and $\cA$) so that
\[
g(a\cdot e)=g(a)\cdot g(e),\quad
\forall g\in G, x\in X, a\in\cA_x, e\in E_x.
\]
We write $K_G^{\cA}(X)$ for the Grothendieck group of the 
additive category of $G$--$\cA$ bundles.
When $G$ acts on $A$ and $\cA$ is the trivial algebra bundle $X\times A$,
we will write $K_G^A(X)$ for $K_G^{\cA}(X)$.

The group $K_{G}^{\cA}(X)$ is contravariantly functorial in the 
variables $G$ and $X$. It is also covariant in $\cA$ (and contravariant
for finite flat maps $\cA\to\cA'$).
\end{defn}

The groups $K_G^{\cA}(X)$ are an equivariant version of twisted
$K$-theory \cite{Ktwist} \cite{Crocker-Raeburn}.  
The prototype of this construction is when
$\cA$ is an algebra bundle with fiber $M_n(\C)$ and $G$ acts 
trivially on $\C$. In this case, Atiyah and Segal \cite{AS2} have 
shown that $G$-algebra bundles are classified
up to Morita equivalence by their 
class in the {\it equivariant Brauer group} $\textrm{Br}_G(X)$, 
which is the torsion subgroup of $H^3(EG\!\times_GX,\Z)$.


\begin{examples}\label{G-A bundles}
Suppose that $\cA$ is a trivial bundle with fiber $A$.

a) When 
$G$ acts trivially on $A$, a $G$--$\cA$ bundle is just an
$A$-linear $G$-bundle. If $A$ is a finite simple algebra, 
$K_G^{\cA}(X)$ is the usual equivariant $K$-group 
$KO_G(X)$, $KU_G(X)$ or $KSp_G(X)$, depending on $A$.

b) When $\cA$ is the trivial bundle with fiber $\C$, and $G$ is
the cyclic group $\mathrm{Gal}(\C/\R)$, a $G$--$\cA$ bundle is 
the same thing as a
Real vector bundle in the sense of Atiyah \cite{A}, 
and our $K^{\C}_G(X)$ is Atiyah's $KR(X)$.

c) When $\cA$ is the trivial bundle with fiber $\H$, and $G$ is a
finite subgroup of $\H^\times$ acting by inner automorphisms on $\H$,
the notion of $G$--$\cA$ bundle seems new. We call these
{\it Real quaternionic bundles}; see Examples \ref{ex:AG} 
and \ref{quaternionic}. We will see 
in Theorem \ref{GRA} and Example \ref{ex:GRA}
that this case is related to the Grothendieck-Witt group of 
skew-symmetric bilinear forms on vector bundles over $X$.

d) (Morita invariance). If $G$ acts on $\cA$ then $G$ acts slotwise on 
$M_n(\cA)$, and the Morita equivalence of $\cA$-bundles and
$M_n(\cA)$-bundles extends to an equivalence between the categories
of \GA bundles and $G$--$M_n(\cA)$ bundles.
Thus $K^{\cA}_G(X)\cong K^{M_n(\cA)}_G(X)$.
\end{examples}

e) Suppose that $G$ acts on $A$, and acts trivially on $X$,
so that $G$ acts on the trivial algebra bundle $\cA$.
Recall that the twisted group ring $A\rtimes G$ is the 
left $A$-module $A\times G$ with multiplication 
$g\cdot a = g(a)\cdot g$. In this case, 
we can form the trivial algebra bundle $\cA\rtimes G$
with fiber $\cA\rtimes G$, and a \GA bundle on $X$ is
the same as an $A\rtimes G$-linear bundle. Indeed,
a left $A\rtimes G$-module $E$ is the same as a left 
$A$-module, with an action of $G$, satisifying the
intertwining relation of Definition \ref{def:GA}.
\smallskip

Many properties of ordinary vector bundles remain valid for \GA bundles.
For example, the kernel of a surjection $E\map{s} E''$ of $G$--$\cA$ bundles
is the subbundle whose fiber at $x$ is the kernel of $E_x \to E''_x$.

\begin{lemma}\label{splitting}
Let $E\map{s} E''$ be a surjection of \GA bundles. Then
the kernel $E'$ of this map is a \GA bundle, and 
\[
0\to E'\to E\map{s} E''\to 0
\]
is a split exact sequence.
\end{lemma}

\begin{proof}
Clearly, 
$E'$ is a \GA subbundle of $E$. To split
the short exact sequence, choose an arbitrary $\cA$-bundle splitting
$E''\to E$; this may be done locally on $X$ and the splittings may
be combined using a partition of unity, as in the classical setting.
Using a Haar measure on $G$, we can average this splitting 
to get an $G$-equivariant $\cA$-module splitting $t$. 
Since $t\circ s$ is idempotent
with kernel $E'$, the sequence is split exact.
\end{proof}

When $X$ is a point and $G$ acts trivially on $A$,
a \GA module is just a finite $A[G]$-module. If $A=\C$
then, as in Example \ref{G-A bundles}(a),
$K^{\C}_G(X)$ is the character ring $R(G)$ of $G$. 
This shows that a \GA bundle need not be a summand 
of a fixed reference bundle $X\times A^n$. 

We will show that every bundle is a summand 
of a different kind of ``trivial'' bundle.
In our setting we define a ``trivial'' \GA bundle to be a bundle of
the type $\cA\oo M$, where $M$ is a finite dimensional 
$\R[G]$-module and $G$ acts diagonally. 
The following theorem is borrowed from 
Segal's paper on equivariant $K$-theory (\cite[p.\,134]{Segal}).

\begin{theorem}\label{summand}
Let $E$ be a $G$--$\cA$ bundle on a compact space. Then
there exists a $G$--$\cA$ bundle $F$ such that $E\oplus F$ is
isomorphic to $\mathcal{A}\oo M$ for some finite
dimensional $G$-module $M.$
\end{theorem}

\begin{proof}
Let $\Gamma =\Gamma(E)$ be the topological space of continuous sections
of $E$; it is naturally a $G$-module. Let $\Gamma'$ denote the union
 of its finite dimensional $G$-invariant subspaces.
By a variant of the Peter-Weyl theorem, $\Gamma'$ is a dense invariant
subspace of $\Gamma$.
Now, for each point $x$ of $X,$ we choose a finite set $s^x_{i}$ of 
global sections such that the $(s^x_{i})(x)$ generate $E_{x}$
as an $\cA_{x}$-module. Since $\Gamma'$ is dense, we may choose the 
$s^x_{i}$ in $\Gamma'$. By continuity, there is an open neighbourhood
$U_{x}$ of $x$ such that the $(s_{i})(y)$ generate $E_{y}$ as an 
$\cA_{y}$-module for any $y\in U_{x}$. Since $X$ is compact, 
we only need the $s_i^x$ for finitely many $x$; they all lie in
a fixed finite dimensional $G$-invariant subspace $M$ of $\Gamma$. 
Thus there is a surjection $\cA\oo M\to E$.
We conclude thanks to Lemma \ref{splitting}.
\end{proof}

The usual argument shows that $K_{G}^{\cA}$ is a homotopy functor.
Let $I$ be the unit interval; if $X$ is a $G$-space, we regard
$X\times I$ as the $G$ space with $g(x,t)=(g(x),t)$.
Given an algebra bundle $\cA$ on $X$, we abusively write $\cA$ for
the pullback of $\cA$ along the projection $p:X\times I\to I$.

\begin{theorem}
If $X$ is a compact space, the projection $X\times I\map{p}X$ 
induces an isomorphism%
\begin{equation*}
p^*: K_{G}^{\cA}(X)\cong K_{G}^{\cA}(X\times I).
\end{equation*}
\end{theorem}

\begin{proof}
The proof is analogous to the one in the classical case 
\cite[I.7.1]{MKbook}: if $p=p(t)$ ($t\in I$) is a continuous family 
of projection operators on a trivial bundle $\cA\oo M,$
it is enough to show that $p(1)$ is isomorphic to $p(0).$ 
By compactness of $X$, each $t\in I$ has a neighborhood $U$
such that for each $u\in U$ the operator $g=1-p(t)-p(u)+2p(t)p(u)$ 
is invertible and $gp(u)g^{-1}=p(t)$.
We conclude by using the compactness of $I$.
\end{proof}

We can use Theorem \ref{summand} in order to prove an analogue of
the Serre-Swan theorem in this framework. If $B$ is a Banach algebra with a 
continuous $G$ action, we define the equivariant Grothendieck group 
$K_{G}(B)$ to be the Grothendieck group of the category 
of finitely generated 
projective $B$-modules with a 
continuous $G$ action. As in the case of bundles, we assume that 
the actions of $B$ and $G$ are intertwined: we have the relation
\begin{equation*}
g(b\cdot e)=g(b)\cdot g(e).
\end{equation*}%
Note that if $G$ is finite then this definition is purely algebraic.

If $E$ is a \GA bundle, and $B=\Gamma(\cA)$ is the Banach algebra 
of sections of the algebra bundle $\cA$, then 
Theorem \ref{summand} implies that the space of sections 
$\Gamma(E)$ is a finitely generated projective $B$-module 
with a continuous $G$-action.

\begin{theorem}\label{Swan}
If $X$ is compact, the functor $\Gamma$ defines an equivalence 
between the category of \GA bundles on $X$ and the category of 
finitely generated projective $B$-modules with a continuous $G$ action.
\end{theorem}


\begin{proof}
This is completely analogous to the usual proof of the Serre-Swan
theorem \cite[I.6.18]{MKbook}. The most difficult step is to show that 
$\Gamma$ is essentially surjective; this is a direct consequence of 
Theorem \ref{summand}.
\end{proof}

\begin{theorem}\label{K(BxG)}
Let $G$ be a finite group and $B=\Gamma(\cA)$.
If $X$ is compact, $\Gamma $ defines an equivalence between 
the category of \GA bundles on $X$
and the category of finitely generated projective modules over 
$B\rtimes G$. Thus
\[ K^{\cA}_G(X) \cong K_0(B\rtimes G). \]
\end{theorem}

\begin{proof}
It is easy to show via an averaging process that the category of 
finitely generated projective $B$-modules with a continuous $G$-action 
is equivalent to the category 
of finitely generated projective  $B\rtimes G$-modules.
\end{proof}

When a compact Lie group $G$ acts on a $C^*$-algebra $B$, 
Julg \cite{Julg} showed that the equivariant $K$-theory of $B$, 
$K_G(B)$, is canonically isomorphic to $K_0(B\rtimes G)$,
where $B\rtimes G$ is the twisted group ring.
Taking $B=\Gamma(\cA)$ yields a more general version of
Theorem \ref{K(BxG)}.

\begin{remm}
When $G$ is a finite group, $A$ is a finite dimensional algebra with 
a $G$-action and $\cA=X\times A$ is the trivial algebra bundle,
then $B$ is $C(X)\oo A$,
$C(X)$ being the ring of continuous functions on $X$. 

In particular, if $G$ acts trivially on $A$ then the usual category 
of $A$-linear $G$-bundles is equivalent to the category of finitely 
generated projective $(C(X)\rtimes G)\oo A$-modules, which is
equivalent to the category of $A$-linear $G$-bundles; 
see Example \ref{G-A bundles}(a).
\end{remm}

Part (a) of the following purely algebraic theorem was originally 
proven by G.\,K.\ Pedersen
\cite[Thm.\,35]{Pedersen}, using the notion of exterior equivalence 
of actions of a (locally compact) group on a $C^*$-algebra.

\begin{theorem}\label{inner}
Suppose that a group $G$ acts on a ring $A$ by inner automorphisms 
$g(a) = x(g)\,a\,x(g)^{-1}$ via a represention $G\map{x} A^\times$.
Then 

\noindent a) 
the twisted algebra $A\rtimes G$ is isomorphic to $A[G]$;

\noindent b)
For every $G$-algebra $C$, 
$(C\oo A)\rtimes G \cong (C\rtimes G)\oo A$.

\noindent c)
If $G$ is a finite group and $A$ is a finite dimensional Banach algebra,
$K^A_G(X)$ is the equivariant K-theory of $A$-bundles
of Example \ref{G-A bundles}(a).
In particular, $K^A_G(X)$ is independent of the representation $x$.
\end{theorem}

\begin{proof}
Fix $g\in G$ and set $x=x(g^{-1})=x(g)^{-1}$, $y(g)=g\cdot x$.
The element $y(g)$ of $A\rtimes G$ commutes with every element of $A$: 
\[
y(g)\cdot a = gx\cdot a\cdot x^{-1}x = g\,(g^{-1}a g)x = a\,y(g).
\]
The $A$-module map $A[G]\to A\rtimes G$ sending $g$ 
to $y(g)$ is multiplicative:
\[
y(g)\cdot y(h) = g\cdot x(g)^{-1}y(h) = gy(h)x(g)^{-1}
= g\, h\cdot x(h)^{-1}x(g)^{-1} = y(gh).  
\]
Thus $y$ defines a ring isomorphism $A[G]\map{\cong} A\rtimes G$. 
This proves (a).

For (b), the map $(C\rtimes G)\oo A\to (C\oo A)\rtimes G$
sending $cg\oo a$ to $(c\oo a)y(g)$ is an isomorphism,
because $y(g)c=g(c)y(g)$ in $(C\oo A)\rtimes G$.

When $G$ is finite and $\dim(A)<\infty$, Theorem \ref{K(BxG)}
and part (b) imply that 
$K^A_G(X)=K_0((C(X)\rtimes G)\oo A)$.
That is, every \GA bundle (with $\cA$ trivial)
has the form $E\oo_A A[G]$ for an $A$-bundle $E$, and
$K^A_G(X)$ is isomorphic to the usual equivariant $K$-theory
of $A$-bundles on $X$.
\end{proof}

\begin{example}\label{ex:K_G}
Suppose that $A$ is a finite dimensional simple $\R$-algebra and that
$G$ acts as the identity on the center of $A$;
This is the case for example if the center is $\R$.
By the Noether-Skolem theorem, every automorphism on $A$ is inner,
i.e., conjugation by an element $x$. If the action lifts to a
representation $G\to A^\times$, as in Theorem \ref{inner},
then $K^A_G(X)$ is either $KO_{G}(X),KU_{G}(X)$ or 
$KSp_{G}(X)$, depending on whether $A$ is a matrix ring 
over $\R$, $\C$ or $\H$.
\end{example}

\begin{remm}
The assumption that $x(g)x(h)=x(gh)$ is needed in Theorem \ref{inner}. 
For instance, when $G=\{1,\tau\}$ is acting trivially on $X$,
with $A=\H$ and $x(\tau)=i$, we show in Example \ref{quaternionic} below
that $K_G^{\H}(X)$ is isomorphic to $KU(X)$ and not $KSp(X)$.
\end{remm}

\medskip\goodbreak
\section{Real $A$-bundles}

We now restrict to the case when $G$ is the cyclic group $\{1,\tau\}$ 
of order~2, that $A$ is a Banach $\R$-algebra, and $\cA$ is an
algebra bundle on $X$ with fiber $A$. If there is an 
involution $\tau$ on $\cA$ which restricts to algebra isomorphisms
$\cA_x\map{\simeq}\cA_{\bar x}$, we call $\cA$ a
{\it Real algebra bundle}. (This is a special case of
$G$ acting on $\cA$ in the sense of the previous section.)

We use the term ``Real $\cA$-bundle'' for a \GA bundle when $G=\{1,\tau\}$.
Here is a paraphrase of Definition \ref{def:GA} in this setting.


\begin{defn}
Suppose that $X$ is a space with involution $\tau$, and 
$\cA$ is a Real algebra bundle on $X$. 
By a {\it Real $\cA$-bundle} on $X$ we mean an $\cA$-bundle $E$ 
together with an involution $\tau:E\to E$ which sends $E_x$ to 
$E_{\bar x}$ and which is twisted $\cA$-linear in the sense that 
$\tau(a\cdot e)=\overline{a}\tau(e)$ for $a\in \cA_x$, $e\in E_x$.
A morphism $\phi:E\to F$ of Real $\cA$-bundles is a morphism of the
underlying $\cA$-bundles commuting with the involution ($\phi\tau=\tau\phi$).

Real $\cA$-bundles form an additive category under Whitney direct sum 
of bundles, and we write $KR^{\cA}(X)$ for its Grothendieck group.
This group is contravariant in $X$ and covariant in $\cA$: 
given $\cA\to \cA'$, the functor $E\mapsto \cA'\oo_{\cA}E$ defines a map
$KR^{\cA}(X)\to KR^{\cA'}(X)$.
Forgetting the involution yields a functor $KR^{\cA}(X)\to K^{\cA}(X)$.
\end{defn}

\begin{example}
If $A$ is equipped with an algebra involution $a\mapsto \bar{a}$,
and $\cA$ is the trivial algebra bundle $X\times A$ with
$\tau(x,a)=(\bar x,\bar a)$, we call $\cA$
a {\it trivial Real algebra bundle}, and use the term
{\it Real $A$-bundle} for a {\it Real $\cA$-bundle}.
Unless stated otherwise, every algebra bundle in the rest of
this section will be a trivial Real algebra bundle.

For example, when the involution on $A$ is trivial, a Real $A$-bundle
on $X$ is just an $A$-linear $G$-bundle, where $G=\{1,\tau\}$. 
As pointed out in Example \ref{G-A bundles}(a), $KR^A(X)$
is $KO_G(X)$, $KU_G(X)$ or $KSp_G(X)$ when $A$ is
a matrix algebra over $\R$, $\C$ or $\H$, respectively.

When $A=\C$ and the involution is complex conjugation,
a Real $A$-bundle is a Real vector bundle in Atiyah's sense
\cite{A}. As pointed out in Example \ref{G-A bundles}(b), 
$KR^A(X)$ is Atiyah's $KR(X)$.
This example motivates our choice to adopt Atiyah's notation $KR$
from \cite{A}.
\end{example}

\begin{variant}\label{A[G]}
Suppose that $A^0$ is a Banach algebra and $A\!=\!A^0[G]$ (with 
$\overline{a+b\tau}=a-b\tau$). 
If $G$ acts trivially on $X$, the category of Real $A$-bundles 
is equivalent to the category of $A^0$-bundles, so
$KR^A(X)\cong K^{A^0}(X)$.
\end{variant}


\begin{lemma}\label{GA-mod}
Fix $A$ and $X$ as above.
There is a faithful functor $\cE$ from the category of 
finitely generated projective left $A\rtimes G$-modules to the category of 
Real $A$-bundles whose underlying $A$-bundle is trivial.
\end{lemma}

\begin{proof}
Suppose that $M$ is a finitely generated projective left $A\rtimes G$-module.
Then $M$ is a finitely generated projective left $A$-module, 
endowed with an involution $\tau\cdot m =\overline{m}$ such that 
$\overline{a.m}=\overline{a}\cdot\overline{m}$. 
Then $\cE(M)=X\times M$ is trivial as an $A$-bundle, and the involution
$\tau(x,m)=(\bar{x},\bar{m})$ makes it a Real $A$-bundle.
As $\cE(M)$ is natural in $M$, $\cE$ is a functor.
\end{proof}

\begin{remm}
Set $B\cong C(X)\oo A$, where
$C(X)$ is the ring of continuous functions on $X$.
By Theorem \ref{K(BxG)}, the category of Real $A$-bundles
on $X$ is equivalent to the category of finitely generated projective
modules over $B\rtimes G$. In particular, $KR^A(X)\cong K_0(B\rtimes G)$.
\end{remm}

The Real $A$-bundle $\cE(A\!\rtimes\!G)$ is 
$X\!\times\!(A\!\rtimes\!G)$, endowed with the involution 
$\tau(x,a+b\tau)=(\bar{x},\bar{a}+\bar{b}\tau)$.
Given a morphism $\phi:\cE(A\!\rtimes\!G)\to E$ of Real $A$-bundles,
$e(x)=\phi(x,1)$ is a section of $E$. We immediately obtain:

\begin{corollary}\label{free}
Given a Real $A$-bundle $E$, a section $e$ of the underlying 
bundle uniquely determines a morphism $\phi:\cE(A\!\rtimes\!G)\to E$
of Real $A$-bundles,  by the formula 
$(x,a+b\tau) \mapsto (x, ae_x + b \bar{e}_{\bar x})$.
\end{corollary}

The universal property in Corollary \ref{free} 
justifies the terminology that Real $A$-bundles can be free.

\begin{defn}
We say that a Real $A$-bundle is {\it free} if it is a 
direct sum of copies of $\cE(A\rtimes G)$, i.e., $\cE(F)$
for a free $A\!\rtimes\!G$-module $F$.
\end{defn}

\begin{lemma}\label{summand free}
If $X$ is compact, any Real $A$-bundle $E$ 
is a direct summand of a free Real $A$-bundle.
\end{lemma}

\begin{proof}
By Theorem \ref{summand}, $E$ is a summand of a Real bundle of the form
$X\times(A\oo \R[G]^n)=\cE(A\rtimes G)^n$.
\end{proof}

\begin{example}\label{ex:RO(G)}
The ring $A\rtimes G$ has two orthogonal idempotents,
$e_+=(1+\tau)/2$ and $e_-=(1-\tau)/2$. Both $Ae_{+}$ and $Ae_-$
are left ideals of $A\rtimes G$, and $A\rtimes G=Ae_+\oplus Ae_-$.
Thus both Real $A$-bundles $\cE(Ae_{+})$ and $\cE(Ae_{-})$ have 
$X\times A$ as their underlying $A$-bundle, and
\[ \cE(A\rtimes G)=\cE(Ae_{+})\oplus\cE(Ae_{-}). \]
The bundle $\cE(Ae_{+})$ has the usual involution 
$\tau(x,a)=(\bar{x},\bar{a})$, while the bundle $\cE(Ae_{-})$ has 
the involution $\tau(x,a)=(\bar{x},-\bar{a})$.

We will write $A_\sigma$ for the left $A\rtimes G$-module $Ae_{-}$,
i.e., the left $A$-module $A$ with the involution $\tau(a)= -\bar{a}$.
Alternatively, $A_\sigma$ is the Real $A$-module $A\oo\R_\sigma$,
where $\R_\sigma$ is the sign representation of $G$. 
\end{example}


%

In the rest of this section, we identify $KR^A(X)$ in some special cases.
Recall that the cyclic group $G=\{1,\tau\}$
acts on a Banach algebra $A$.

\smallskip\goodbreak
\begin{center}{\bf Finite simple algebras}\end{center}
\smallskip

We will be primarily interested in Real $A$-bundles when $A$ is a
finite dimensional simple $\R$-algebra. 
Studying Real $A$-bundles for any finite dimensional
{\it semisimple} $\R$-algebra $A$ does not yield more generality.
Since any finite dimensional semisimple $\R$-algebra with involution $A$
is a product of simple algebras with involution $A_i$ and algebras
$A_j[G]$, every Real $A$-bundle is canonically a product of Real
$A_i$-bundles and Real $A_j[G]$-bundles. 
We leave the details to the reader.

\begin{lemma}\label{central simple}
If $A$ is a central simple $\C$-algebra, 
and the involution is trivial on $\C$,
then $KR^A(X)\cong KU_G(X)$.
\end{lemma}

\begin{proof}
By Noether-Skolem, $\bar{a}=xax^{-1}$ for 
some $x\in A^\times$ with $x^2\in\C$.
Let $c\in\C$ be a square root of $x^2$; then 
$\bar{a}=(x/c)a(x/c)^{-1}$ and $(x/c)^2=1$.
The result now follows from Theorem \ref{inner}(b).
\end{proof}

\begin{example}\label{ex:C11}
Let $A$ be the algebra $M_{2}(\R)$ with the involution defined by
conjugation by the diagonal matrix $j=(1,-1)$.
That is, $A$ is the Clifford algebra $C^{1,1}$ with
$i=\left(\atop01\atop{-1}0\right)$ satisfying $i^2=-1$ and 
$\bar{i}=-i$, while
$k=\left(\atop01\atop10\right)$ satisfyies $k^2=+1$ and
$\bar{k}=-k$.
Since $j^2=1$, Theorem \ref{inner}(b) yields 
$KR^A(X)\cong KO_G(X)$.
\end{example}

\begin{lemma}\label{ss}
If $A$ is a finite simple $\R$-algebra, 
$A\rtimes G$ is a semisimple $\R$-algebra.
\end{lemma}

\begin{proof}
If $V$ is a minimal left ideal of $A$, $\tau(V)$ is either $V$
or disjoint from $V$ by Schur's Lemma. In the latter case, consider
the $A\rtimes G$-modules $M=V\oplus\tau(V)$; either it is simple or
has the form $W\oplus W'$ where $W=\tau(W)$ and $W'=\tau(W')$.
Since $A$ is a direct sum of minimal left ideals,
this shows that $A$ is a direct sum of simple left 
$A\rtimes G$-modules, and hence that $A\rtimes G$ is semisimple.
\end{proof}

\begin{examples}\label{ex:AG}
If $A=\C$ and the involution is complex conjugation
(resp., trivial) then $A\rtimes G$ is $M_2(\R)$ (resp., $\C\times\C$).

If $A=\H$ and the involution is conjugation by $i$, then
$A\rtimes G$ is $M_2(\C)$. In this case, we call a Real $\H$-bundle
a {\it Real quaternionic bundle} on $X$; these came up in \cite{KWpre}.
We will show in Example \ref{quaternionic} below
that if $G$ acts trivially on $X$ then
the group $KR^A(X)$ is $KU(X)$.
\end{examples}



\smallskip\goodbreak
\begin{center}{\bf Balanced algebras}\end{center}
\smallskip

Another family of examples comes from the observation that every
$\R$-algebra $A$ with an algebra involution $\tau$ is a $\Z/2$-graded
algebra, with $A_0=A^\tau$ and $A_1 =\{ a\in A:\tau(a)=-a\}$). 

\begin{defn}
We will say that an algebra involution is {\it balanced} 
(or that $A$ is balanced) if there is a unit $u$ in $A$ with
$\bar{u}=-u$. This implies that
there is a left $A_0$-module isomorphism $A_0\cong A_1$, $a\mapsto au$.
\end{defn}

\begin{remm}
If $A$ is simple, and there is a left $A_0$-module isomorphism 
$A_0\cong A_1$ sending $1$ to $u$, then $u$ must be a unit of $A$.
Indeed, for each $b\in A_0$ there is a unique $a\in A_0$ such that
$au=ub$. It follows that $uA_1=uA_0u=A_1u$ and hence that $Au$ is
a 2-sided ideal of $A$. Since $A$ is a simple algebra, we must 
have $Au=A$ and hence $u$ is a unit of $A$.
\end{remm}

The prototype of a balanced involution is the canonical involution 
(induced by $-1$ on $V$) on the Clifford algebra $A=C^{p,q}$ of a 
quadratic form on $V$ of rank $p+q\ge2$ and signature $q-p$. 
As in Example \ref{ex:C11}, $A^0$ may not be simple.

\begin{example}
$A$ cannot be balanced if $\dim(A)$ is odd,
and may not be balanced if $\dim(A)$ is even. For example,
the algebra $A = M_4(\R)$, with the involution defined 
by conjugation with the diagonal matrix $(1,1,1,-1)$, 
is not balanced: $A_0=M_3(\R)\times\R$ and 
$A_1$ is the 6-dimensional subspace spanned by 
$\{e_{4j}, e_{j4}: j<4\}$.
By Theorem \ref{inner}, $KR^A(X)\cong KO_G(X)$.
\end{example}


When $A$ is balanced, the map $a\mapsto au$
defines an isomorphism $A\map{\cong}A_\sigma$ of Real $A$-modules,
where $A_\sigma$ is defined in Example \ref{ex:RO(G)}.


%

\begin{theorem}\label{thm:balanced}
Let $A$ be a Banach algebra with a balanced involution. 
If $X$ is compact, any Real $A$-bundle $E$ on $X$ is a direct summand 
of a Real $A$-bundle of the form 
$X\times A^n$, 
with involution $\overline{(x,a)}=(\bar{x},\bar{a})$.

The Grothendieck group $KR^A(X)$ is isomorphic to $K_0(\Lambda)$,
where $\Lambda$ is the ring of continuous functions $f:X\to A$
satisfying $f(\bar{x})=\overline{f(x)}$.

Finally, if the involution on $X$ is trivial, 
$KR^A(X)$ is the usual Grothendieck group $K^{A_0}(X)$ 
of $A_0$-bundles on $X$.

\end{theorem}

\begin{proof}
By Lemma \ref{summand free}, any Real $A$-bundle is a summand of
a free Real $A$-bundle; by Lemma \ref{GA-mod} and 
Example \ref{ex:RO(G)}, 
free Real $A$-bundles have the form $X\times(A\times A_\sigma)^n$.
Using the isomorphism $A\map{\cong}A_\sigma$ of Real $A$-modules,
it follows from Theorem \ref{summand} that every Real $A$-bundle is
a summand of a trivial Real $A$-bundle $X\times A^m$.

The assertion about $\Lambda$ comes from Theorem \ref{K(BxG)}.
If the involution is trivial on $X$, $\Lambda$ is the ring of
continuous functions $X\to A^0$.
\end{proof}

\begin{example}[Real quaternionic bundles]\label{quaternionic}
Theorem \ref{thm:balanced} applies to Real quaternion bundles
(see Example \ref{ex:AG}). Indeed,  
conjugation by $i$ is a balanced involution of $A=\H$
($\bar{j}=-j$ and $\bar{k}=-k$).

If $X$ has a trivial involution, then
Theorem \ref{thm:balanced} shows that the category of Real quaternionic 
bundles is equivalent to the category of complex vector bundles 
on $X.$ In particular, $KR^{\H}(X) \cong KU(X)$.

On the other hand, if $X$ is $Y\times\{\pm1\}$ 
 with $\overline{(y,\varepsilon)}=(y,-\varepsilon)$,
the group $KR^{\H}(X)$ is isomorphic to the symplectic $K$-theory $KSp(Y)$.
\end{example}

\begin{remm}
In \cite{Dupont}, J.\,Dupont defined ``symplectic bundles'' which,
in spite of the name, are not related to our constructions. 
His groups $Ksp_n(X)$ are just $KR_{n+4}(X)$.
For instance, when $X$ has a trivial involution, our group 
$KR^{\H}(X)$ is $KU(X)$
but Dupont's group is $KR_4(X) = KO_4(X) = Ksp(X)$.
\end{remm}

\begin{example}\label{Clifford}
Every Clifford algebra $C^{p,q}$ is canonically $\Z/2$-graded, 
and the associated involution is balanced. Thus 
Theorem \ref{thm:balanced} applies when $C^{p,q}$ is a simple algebra.
For example, consider the Clifford algebra $C^{0,2}\cong M_2(\R)$;
its fixed subalgebra is $\C$. As in Example \ref{quaternionic},
if $X$ has a trivial involution, then $KR^{A}(X) \cong KU(X)$.
However, if $X=Y\times\{\pm1\}$  
then $KR^{A}(X)$ is isomorphic to $KO(Y)$.
\end{example}

%
%
%
%

\goodbreak
\section{Grothendieck-Witt groups}\label{sec:GR}

Suppose that an $\R$-algebra $A$ is a $*$-algebra
(i.e., $*\!:\!A\map{\cong} A^\op$ is an algebra isomorphism 
such that $a^{**}=a$ for all $a\in A$), and that $\tau$
is an algebra involution of $A$ commuting with $*$.
If $\cA$ is an algebra bundle on $X$ with fibers $A$,
and $E$ is a Real $\cA$-bundle, then the dual $E^*=\Hom_A(E,A)$ 
is also a Real $\cA$-bundle; if $f\in E^*_x$ then $af\in E^*_x$ is  
defined by $e\mapsto f(e)a^*$, and $(\tau f)(e)=\tau(f(e))$.  

\begin{defn}
An {\it $\eps$-symmetric form} $(\eps=\pm1$)
on a Real $\cA$-bundle $E$
is an isomorphism $\psi:E\map{\cong} E^*$ of Real $\cA$-bundles
such that the  bilinear form
$B(x,y)=\psi(x)y$ satisfies $B(y,x)=\eps B(x,y)^*$.
Note that $B(x,ay)=aB(x,y)$ but $B(ax,y)=B(x,y)a^*$.

The Grothendieck-{\!}Witt group ${}_{\eps}GR^{\cA}(X)$ is defined
to be the Grothendieck group of the category of $\eps$-symmetric 
forms $(E,\psi)$. 
\end{defn}

The groups $GR^{\cA}(X)\!=\!{}_{+1}GR^{\cA}(X)$, 
and to a lesser extent ${}_{-1}GR^{\cA}(X)$, 
are the focus of our recent paper \cite{KWpre}.

\begin{remm}
The definition of Grothendieck-Witt group in \cite{Schlichting}
includes the relation that $[(E,\varphi)]=[H(L)]$ whenever 
$E$ has a {\it Lagrangian}, i.e., a subspace $L$ with $L=L^\perp$.
By Lemma \ref{splitting} and \cite[Lemma 2.9]{Schlichting}, 
this relation is redundant in our topological framework.
\end{remm}

Here is a useful general principle. 

\begin{lemma}\label{self-adjoint}
Given a symmetric form $\psi$ with bilinear form $B$,
and an $\eps$-symmetric form $\varphi$ on $E$, 
$\theta=\psi^{-1}\varphi$ is $\eps$-self-adjoint for $B$, i.e.:
\[ B(\theta x,y) =\eps B(x,\theta y). \]
\end{lemma}

\begin{proof}
We have
$B(\theta x,y)\!=\psi(\psi^{-1}\varphi x)y=\!\varphi(x)y$, 
and similarly $B(x,\theta y)$ equals 
$B(\theta y,x)^* = \left[\varphi(y)x\right]^* = \eps\varphi(x)y$.
\end{proof}

We can extend $*$ to $A\oo\C$ by setting 
$(a\oo z)^*=(a^*)\oo\bar z$, where 
$\bar z$ is the complex conjugate of $z$.
There is a second anti-involution $\dag$ on $A\oo\C$,
defined by $(a\oo z)^\dag=(a^*)\oo z$.
If $E$ is  a finitely generated projective $A\oo\C$-module,
we write $E^*$ and $E^\dag$ for its dual $\Hom_{A\oo\C}(E,A\oo\C)$,
endowed with the respective anti-involutions $*$ and $\dag$.
These left module structures define $(a\oo z)f$ to be 
$e\mapsto f(e)(a^*\oo\bar z)$ for $f\in E^*$, and
$e\mapsto f(e)(a^*\oo z)$ for $f\in E^\dag$.

A {\it Hermitian form} on $E$ is a map $\psi:E\to E^*$ which
is symmetric for the anti-involution $*$. That is,
its associated form $\langle x,y\rangle\!=\!\psi(x)(y)$ satisfies
$\langle y,x\rangle = \langle x,y\rangle^*$ and
$\langle (a\oo z)x,y\rangle=\langle x,y\rangle(a^*\oo\bar z)$.

\begin{lemma}\label{r:self-adjoint}
If $\varphi$ is an $\eps$-symmetric form on $E$ for $\dag$ 
and $\psi$ is Hermitian
then $\theta=\psi^{-1}\varphi$ is $\C$-antilinear, and 
$\langle\theta x,y\rangle=\eps\;\overline{\langle x,\theta y\rangle}$,
where $\overline{a\oo z}$ denotes $a\oo\bar{z}$.
\end{lemma}

\begin{proof}
The proof is similar to that of Lemma \ref{self-adjoint}:
$\langle\theta x,y\rangle=\varphi(x)y$ as before, and
\[
\langle x,\theta y\rangle = \langle\theta y,x\rangle^*
= [\varphi(y)x]^* = \eps[\varphi(x)y]^{\dag*}
=\eps \overline{\langle\theta x,y\rangle}. 
\]
The fact that $\theta$ is $\C$-antilinear follows from the
$\C$-linearity of $\varphi$ and the $\C$-antilinearity of $\psi$.
\end{proof}


We first consider the following class of examples.
Recall from \cite{Schroder} that a $C^*$-algebra over $\R$ is a
Banach $*$-algebra over $\R$, $*$-isometrically isomorphic to a
norm-closed subalgebra of linear operators on a real Hilbert space.
We extend the $C^*$-structure of $A$ (as an $\R$-algebra) to
a  $C^*$-structure of $A\oo\C$ (as a $\C$-algebra)
by setting $(a\oo z)^*=(a^*)\oo\bar z$.

\begin{example}\label{C*}
Suppose that $A^0$ is a $C^*$-algebra over $\R$.
The complex $C^*$-algebra $A=A^0\oo\C$ is equipped with the
involution $\tau$ coming from complex conjugation. The structure map 
$A\to A^\op$ is $(a_0\oo z) \mapsto a_0^*\oo z$
(the $\dag$ map defined before Lemma \ref{r:self-adjoint});
as it commutes with $\tau$, 
we can consider the groups ${}_{\eps}GR^A(X)$.

In this case we consider the two auxilliary $\R$-algebras
\[
A^+=A^0\oo_{\R}M_2(\R) \quad and\quad A^-=A^0\oo_{\R}\H.
\]
That is, $A^+$ is generated by $A=A^0\oo\C$ and an element $j$
satisfying $j^2=1$ and $ij=-ji$, while $A^-$ is generated by $A$ and
an element $j$ satisfying $j^2=-1$ and $ij=-ji$.
In both cases, $j$ commutes with $A$.  The involution on
$A$ induces an involution on $A^{\pm}$ fixing $A^0$ and $j$.

If $\cA^0$ is an algebra bundle with fiber $A^0$, we can define
algebra bundles $\cA$ and $\cA^{\pm}$ in an obvious way.
\end{example}

\begin{theorem}\label{GRA}
If $\cA=\cA^0\oo\C$, with involution $\tau(a\oo z)=a\oo\bar z$,
and anti-involution $(a\oo z)\mapsto a^*\oo z$,
we have canonical isomorphisms%
\begin{equation*}
_{+1}GR^{\mathcal{A}}(X)\cong KR^{\mathcal{A}^{+}}(X)
\quad and \quad
_{-1}GR^{\mathcal{A}}(X)\cong KR^{\mathcal{A}^{-}}(X).
\end{equation*}
\end{theorem}

\begin{proof}
We give the proof for $\eps=+1$; the proof for $\eps=-1$ is similar. 
Let $E$ be a Real $A$-bundle provided with a symmetric bilinear form $\varphi$.
Choose a Hermitian metric $\psi$ on $E$ compatible with the 
involution. (Hermitian metrics exist by \cite[2.7]{K80};
$(\psi(x)+\psi(\tau x))/2$ is compatible with $\tau$.)
Then $\varphi$ is associated to a self-adjoint invertible operator 
$\theta$ on $E$ by Lemma \ref{self-adjoint} and 
Lemma \ref{r:self-adjoint};
by construction, $\theta$ commutes with the involution and 
is complex conjugate linear.

As it is self-adjoint, all eigenvalues of $\theta$ are real and positive.
Changing the metric $\psi$ up to homotopy, we may even assume that
$\theta$ is unitary. Since $\theta=\theta^*$ we have $\theta^2=1$.
Setting $j=i\theta$, we have $j^2=1$ and $ij=-ji$ 
(as $i\theta=-\theta i$). Thus $E$ is a Real $\cA^+$ bundle.

Conversely, given a Real $\cA^+$ bundle $E$ on $X$, choose
a $G$-invariant Hermitian metric $\psi$ on $E$ whose bilinear form
satisfies $\langle jx,y\rangle=\langle x,jy\rangle$. 
Setting $\theta=ji$, $\varphi =\psi\theta$ is a symmetric 
bilinear form. By inspection, the map between $K$-groups so 
obtained is inverse to the map defined in the first paragraph.
\end{proof}

\smallskip

\begin{corollary}\label{GRA(X)}
The theory $KR^{\cA^{+}}(X)$ is canonically isomorphic to the usual
equivariant $K$-theory $K_{G}^{A^0}(X),$ where $G$ is acting trivially on
$\cA^{0}$ Therefore, we have an isomorphism%
\begin{equation*}
GR^{\cA^+}(X)\cong K_{G}^{A^0}(X).
\end{equation*}%
\end{corollary}

\begin{proof}
For $\eps=+1$, the involution on $A^+=A^0\oo_{\R}M_2(\R)$ 
is conjugation by the diagonal matrix $j=(1,-1)$. Exactly as in
Example \ref{ex:C11}, Theorem \ref{inner} yields 
$KR^{\cA^{+}}(X)\cong K_{G}^{A^{0}}(X).$
\end{proof}

Despite the symmetry of Theorem \ref{GRA}, the theory for $\eps=-1$
is quite different for $\eps=+1$. Indeed,
the algebra $\cA^{-}$ is $\cA^{0}\oo_{\R}\H$,
where $G$ acts trivially on $\cA^{0}$ and by the quaternionic involution
($i\mapsto -i,j\mapsto -j$) on $\H$. That is, $\tau$ is conjugation by
$k=ij$. Note that Theorem \ref{inner} 
does not apply because $k^2\ne1$.
Indeed, the group $KR^{\cA-}(X)$ is different from $KSp_G(X)$ in general. 

\begin{example}\label{ex:GRA}
The group $GR(X)$ discussed in \cite{KWpre} is the special case
$GR^{\C}(X)$, where $G$ acts on $A=\C$ by conjugation and $A^0=\R$.
by Corollary \ref{GRA(X)} we have $GR(X)\cong KO_G(X)$.

On the other hand, the group ${}_{-1}GR(X)$ is isomorphic to
$KR^{\H}(X)$ by Theorem \ref{GRA}, where $\tau$ is conjugation by $i$
(as in Example \ref{quaternionic}).
This seems to be a new example of a twisted $K$-group.
\end{example}

\smallskip
Now suppose that $A$ is a $C^*$-algebra over $\R$,
with a $G$-action compatible with the $C^*$-structure,
so that we can define ${}_{\eps}GR^A(X)$.

\begin{theorem}\label{realGRA}
If $A$ is a $C^*$-algebra over $\R$, and $G$ acts compatibly with $*$,
we have canonical isomorphisms%
\begin{equation*}
_{\eps}GR^{\cA}(X)\cong KR^{\cA'}(X),
\end{equation*}
where $\cA'$ is the algebra bundle $\cA\oo_{\R}\R[t]/(t^2=\eps)$.
\end{theorem}

Note that $\cA'\cong \cA\times \cA$ if $\eps=+1$, while
$\cA'\cong \cA\oo_{\R}\C$ if $\eps=-1$.

\begin{proof}
We indicate the modifications to the proof of Theorem \ref{GRA}.
Given $(E,\varphi)$, we choose a Riemannian metric $\psi$
(instead of a Hermitian metric), compatible with the involution,
and form the $\eps$-self-adjoint operator $\theta$ 
using Lemma \ref{self-adjoint}. Changing the metric up to homotopy,
we may assume that $\theta$ is orthogonal, so $\theta^2=\eps I$. 
Thus $E$ is a Real $\cA[\theta]$-bundle, where $a\theta=\theta a$.

Conversely, given a Real $A[\theta]$-bundle
$E$, choose a $G$-invariant Riemannian metric $\psi$ on $E$ whose
bilinear form satisfies 
$\langle\theta x,y\rangle=\langle x,\theta y\rangle$.
Then $\varphi=\psi\theta$ is a symmetric form on $E$.
This proves that $GR^{\cA}(X)\cong KR^{\cA[\theta]}(X)$.
\end{proof}

\begin{examples}
(a) if $A=M_n(\R)$ and $*$ is matrix transpose, then 
the involution can be conjugation by an orthogonal matrix $x$.
If $x^2=1$ then $GR^A(X)\cong KR^A_G(X)\cong KO_G(X)$
by Example \ref{ex:K_G}.

(b) if $A=M_n(\C)$ and $*$ is conjugate transpose, then 
the involution can be conjugation by a unitary matrix $x$.
Here $GR^A(X)\cong KU_G(X)$.

(c) if $A=M_n(\H)$ and $*$ is matrix transpose composed with
$\H\cong\H^\op$, then the involution can be conjugation by a 
symplectic matrix $x$. In this case, $GR^A(X)\cong KSp_G(X)$,
again by Example \ref{ex:K_G}.
\end{examples}

\end{document}